\documentclass{amsart}
\usepackage{amsmath}
\usepackage{amsfonts}
\usepackage{indentfirst}

\newtheorem{theorem}{Theorem}
\theoremstyle{plain}

\newtheorem{corollary}{Corollary}

\newtheorem{definition}{Definition}

\newtheorem{lemma}{Lemma}

\numberwithin{equation}{section}

\begin{document}
\title[On GA-Convex Functions via fractional Integral]{Generalized
Hermite-Hadamard-Fejer type inequalities for GA-convex functions via
Fractional integral }
\author{\.{I}mdat \.{I}\c{s}can}
\address{Department of Mathematics, Faculty of Arts and Sciences,\\
Giresun University, 28100, Giresun, Turkey.}
\email{imdati@yahoo.com, imdat.iscan@giresun.edu.tr}
\author{Sercan Turhan}
\address{Dereli Vocational High School,\\
Giresun University, 28100, Giresun, Turkey.}
\email{sercanturhan28@gmail.com, imdat.sercan.turhan@giresun.edu.tr}

\begin{abstract}
In this paper, it is a fuction that is a GA-convex differentiable for a new
identity. As a result of this identity some new and general fractional
integral inequalities for differentiable GA-convex functions are obtained.
\end{abstract}

\maketitle

\section{Introduction}

The classical or the usual convexity is defined as follows:

A function $f:\emptyset\neq I\subseteq 
\mathbb{R}
\longrightarrow 
\mathbb{R}
$, is said to be convex on $I$ if inequality

\begin{equation*}
f\left(tx+\left(1-t\right)y\right) \leq tf(x)+\left(1-t\right)f(y)
\end{equation*}
holds for all $x,y \in I $ and $t\in \left[0,1\right]$.

A number of papers have been written on inequalities using the classical
convexty and one of the most fascinating inequalities in mathematical
analysis is stated as follows 
\begin{equation}
f\left( \frac{a+b}{2}\right) \leq \frac{1}{b-a}\int\limits_{a}^{b}f(x)dx\leq 
\frac{f(a)+f(b)}{2}\text{,}  \label{1.1}
\end{equation}%
where $f:I\subseteq 
\mathbb{R}
\longrightarrow 
\mathbb{R}
$ be a convex mapping and $a,b\in I$ with $a\leq b$ . Both the inequalities
hold in reversed direction if $f$ is concave. The inequalities stated in (%
\ref{1.1}) are known as Hermite-Hadamard inequalities.

For more results on (\ref{1.1}) which provide new proof, significantly
extensions, generalizations, refinements, counterparts, new
Hermite-Hadamard-type inequalities and numerous applications, we refer the
interested reader to \cite{D12012}-\cite{HM94},\cite%
{I113,N00,N03,SYQ13,ZJQ13} and the references therein.

The usual notion of convex function have been generalized in diverse
manners. One of them is the so called GA-convex functions and is stated in
the defination below.

\begin{definition}
\cite{N00,N03}A function $f:I\subseteq 
\mathbb{R}
= \left(0, \infty \right)\longrightarrow 
\mathbb{R}
$ is said to be GA-convex function on $I$ if 
\begin{equation*}
f\left(x^{\lambda }y^{1- \lambda}\right)\leq \lambda
f\left(x\right)+\left(1-\lambda\right)f\left(y\right)
\end{equation*}
holds for all $x,y\in I$ and $\lambda \in \left[0,1\right]$, where $%
x^{\lambda}y^{1-\lambda}$ and $\lambda
f\left(x\right)+\left(1-\lambda\right)f\left(y\right)$ are respectively the
weighted geometric mean of two positive numbers $x$ and $y$ and the weighted
arithmetic mean of $f(x)$ and $f(y)$.

The defination of GA-convexity is further generalized as s-GA-convexity in
the second sense as follows.
\end{definition}

\begin{definition}
\cite{II14} A function $f:I\subseteq 
\mathbb{R}
= \left(0, \infty \right)\longrightarrow 
\mathbb{R}
$ is said to be s-GA-convex function on $I$ if 
\begin{equation*}
f\left(x^{\lambda }y^{1- \lambda}\right)\leq \lambda^{s}
f\left(x\right)+\left(1-\lambda\right)^{s}f\left(y\right)
\end{equation*}
holds for all $x,y\in I$ and $\lambda \in \left[0,1\right]$ and for some $%
s\in \left(0,1\right]$.
\end{definition}

For the properties of GA-convex functions and GA-s-convex function, we refer
the reader to \cite{I114,ML14,ZJQ13,ZCZ10,JZQ113,II14,LDM0415,LDM0315} and
the reference there in.

Most recently, a number of findings have been seen on Hermite-Hadamard type
integral inequalities for GA-convex and for GA-s-convex functions.

Zang at all. in \cite{ZJQ213} established the following Hermite-Hadamard
type integral inequalities for GA-convex function.

\begin{theorem}
\cite{ZJQ213} Let $f:I\subseteq \mathbb{R}_{+}\mathbb{\rightarrow R}$ be
differentiable on $I^{\circ }$, and $a,b\in I$ with $a<b$ and $f^{\prime
}\in L\left[ a,b\right] .$ If $\left\vert f^{\prime }\right\vert ^{q}$ is
GA-convex on $\left[ a,b\right] $ for $q\geq 1$, then%
\begin{equation}
\left\vert bf(b)-af(a)-\int\limits_{a}^{b}f(x)dx\right\vert \leq \frac{\left[
\left( b-a\right) A\left( a,b\right) \right] ^{1-1/q}}{2^{1/q}}
\end{equation}%
\begin{equation*}
\times \left\{ \left[ L(a^{2},b^{2})-a^{2}\right] \left\vert f^{\prime
}(a)\right\vert ^{q}+\left[ b^{2}-L(a^{2},b^{2})\right] \left\vert f^{\prime
}(b)\right\vert ^{q}\right\} ^{1/q}.  \label{1.2}
\end{equation*}
\end{theorem}

\begin{theorem}
\cite{ZJQ213} Let $f:I\subseteq \mathbb{R}_{+}\mathbb{\rightarrow R}$ be
differentiable on $I^{\circ }$, and $a,b\in I$ with $a<b$ and $f^{\prime
}\in L\left[ a,b\right] .$ If $\left\vert f^{\prime }\right\vert ^{q}$ is
GA-convex on $\left[ a,b\right] $ for $q>1$, then%
\begin{equation}
\left\vert bf(b)-af(a)-\int\limits_{a}^{b}f(x)dx\right\vert \leq \left( \ln
b-\ln a\right)
\end{equation}%
\begin{equation*}
\times \left[ L(a^{2q/(q-1)},b^{2q/(q-1)})-a^{2q/(q-1)}\right] ^{1-1/q}\left[
A\left( \left\vert f^{\prime }(a)\right\vert ^{q},\left\vert f^{\prime
}(b)\right\vert ^{q}\right) \right] ^{1/q}.  \label{1.3}
\end{equation*}
\end{theorem}

\begin{theorem}
Let $f:I\subseteq \mathbb{R}_{+}\mathbb{\rightarrow R}$ be differentiable on 
$I^{\circ }$, and $a,b\in I$ with $a<b$ and $f^{\prime }\in L\left[ a,b%
\right] .$ If $\left\vert f^{\prime }\right\vert ^{q}$ is GA-convex on $%
\left[ a,b\right] $ for $q>1$ and $2q>p>0$, then%
\begin{equation}
\left\vert bf(b)-af(a)-\int\limits_{a}^{b}f(x)dx\right\vert \leq \frac{%
\left( \ln b-\ln a\right) ^{1-1/q}}{p^{1/q}}
\end{equation}%
\begin{eqnarray*}
&&\times \left[ L(a^{(2q-p)/(q-1)},b^{(2q-p)/(q-1)})\right] ^{1-1/q} \\
&&\times \left\{ \left[ L(a^{p},b^{p})-a^{p}\right] \left\vert f^{\prime
}(a)\right\vert ^{q}+\left[ b^{p}-L(a^{p},b^{p})\right] \left\vert f^{\prime
}(b)\right\vert ^{q}\right\} ^{1/q}.  \label{1.4}
\end{eqnarray*}
\end{theorem}

\begin{theorem}
Suppose that $f:I\subseteq 
\mathbb{R}
=\left( 0,\infty \right) \longrightarrow 
\mathbb{R}
$ is s-GA-convex function in the second sense, where $s\in \lbrack 0,1)$ and
let $a,b\in \lbrack 0,\infty )$, $a<b$. If $f\in L\left[ a,b\right] $, then
the following inequalities hold 
\begin{equation}
2^{s-1}f\left( \sqrt{ab}\right) \leq \frac{1}{lnb-lna}\int\limits_{a}^{b}%
\frac{f(x)}{x}dx\leq \frac{f(a)+f(b)}{s+1}\text{.}  \label{1.5}
\end{equation}%
the constant $k=\frac{1}{s+1}$ is the best possible in the second inequality
in (\ref{1.1}).
\end{theorem}

If $f$ in Theorem 5 is GA-convex function, then we get the following
inequalities. 
\begin{equation}
f\left( \sqrt{ab}\right) \leq \frac{1}{lnb-lna}\int\limits_{a}^{b}\frac{f(x)%
}{x}dx\leq \frac{f(a)+f(b)}{2}\text{.}  \label{1.6}
\end{equation}

For more results on GA-convex function and s-GA-convex function see e.g \cite%
{I114,II14,LDM0415} and \cite{LDM0315}.

\begin{definition}
A function $f:I\subseteq 
\mathbb{R}
=\left( 0,\infty \right) \longrightarrow 
\mathbb{R}
$ is said to be geometrically symmetric with respect to $\sqrt{ab}$ if the
inequality 
\begin{equation*}
g\left( \frac{ab}{x}\right) =g\left( x\right)
\end{equation*}%
holds for all $x\in \lbrack a,b]$
\end{definition}

\begin{definition}
Let $f \in L[a,b] $. The right-hand side and left-hand side Hadamard
fractional integrals $J_{a^{+}}^{\alpha }f $ and $J_{b^{-}}^{\alpha }f $ of
order $\alpha > 0 $ with $b>a \leq 0 $ are defined by

\begin{equation*}
J_{a^{+}}^{\alpha }f(x)=\frac{1}{\Gamma (\alpha )}\overset{x}{\underset{a}{%
\int }}\left( \ln \frac{x}{t}\right) ^{\alpha -1}f(t)\frac{dt}{t},\ x>a
\end{equation*}%
\begin{equation*}
J_{b^{-}}^{\alpha }f(x)=\frac{1}{\Gamma (\alpha )}\overset{b}{\underset{x}{%
\int }}\left( \ln \frac{t}{x}\right) ^{\alpha -1}f(t)\frac{dt}{t},\ x<b
\end{equation*}%
respectively where $\Gamma (\alpha )$ is the Gamma function defined by $%
\Gamma (\alpha )=\overset{\infty }{\underset{0}{\int }}e^{-t}t^{\alpha -1}$
\end{definition}

\begin{lemma}
For $0< \theta \leq 1 $ and $0< a \leq b $ we have 
\begin{equation*}
\left\vert a^{\theta}-b^{\theta} \right\vert \leq \left( b-a\right)^{\theta}.
\end{equation*}
\end{lemma}

In \cite{H11} D. Y. Hwang found out a new identity and by using this
identity, established a new inequalities. In this paper, we established a
new identity similar to that's identity in \cite{H11} and then we obtained
some new and general integral inequalities for differentiable GA-convex
functions using this lemma.

\section{Main result}

Throughout this section, let $\left\Vert g\right\Vert _{\infty }=\sup_{x\in %
\left[ a,b\right] }\left\vert g(x)\right\vert $, for the continuous function 
$g:[a,b]\longrightarrow 
\mathbb{R}
=\left( 0,\infty \right) \longrightarrow 
\mathbb{R}
$ be differentiable mapping $I^{o}$, where $a,b\in I$ with $a\leq b$, and $h:%
\left[ a,b\right] \longrightarrow \left[ 0,\infty \right) $ be
differentiable mapping. In the following sections, for convenience, let the
notion $L\left( t\right) =a^{t}G^{1-t}$, $U\left( t\right) =b^{t}G^{1-t}$
and $G=G\left( a,b\right) =\sqrt{ab}$.

\begin{lemma}
If $f^{\prime}\in L\left(a,b\right)$ then the following inequality holds:

\begin{equation}  \label{2.1}
\left[ h(b)-2h(a)\right] \frac{f(a)}{2}+h(b)\frac{f(b)}{2}%
-\int\limits_{a}^{b}f(x)h^{\prime }(x)dx\ 
\end{equation}%
\begin{equation*}
=\frac{lnb-lna}{4}\left\{ \int\limits_{0}^{1}\left[ 2h\left(
a^{t}G^{1-t}\right) -h(b)\right] f^{\prime }\left( a^{t}G^{1-t}\right)
a^{t}G^{1-t}dt\ \right.
\end{equation*}%
\begin{equation*}
\ \ \ \ \ \ \ \ \ \ \ \ \ \ \ \ \ \ \left. +\int\limits_{0}^{1}\left[
2h\left( b^{t}G^{1-t}\right) -h(b)\right] f^{\prime }\left(
b^{t}G^{1-t}\right) b^{t}G^{1-t}dt\right\}.
\end{equation*}
\end{lemma}

\begin{proof}
By the integration by parts, we have 
\begin{equation*}
I_{1}=\int\limits_{0}^{1}\left[ 2h\left( a^{t}G^{1-t}\right) -h(b)\right]
d\left( f\left( a^{t}G^{1-t}\right) \right)
\end{equation*}%
\begin{equation*}
=\left. \left[ 2h\left( a^{t}G^{1-t}\right) -h(b)\right] f\left(
a^{t}G^{1-t}\right) \right\vert _{0}^{1}
\end{equation*}%
\begin{equation*}
\ \ \ \ \ \ \ \ \ \ \ \ \ -2\ln \left( \frac{a}{G}\right)
\int\limits_{0}^{1}f\left( a^{t}G^{1-t}\right) h^{\prime }\left(
a^{t}G^{1-t}\right) a^{t}G^{1-t}dt
\end{equation*}%
\newline
and

\begin{equation*}
I_{2}=\int\limits_{0}^{1}\left[ 2h\left( b^{t}G^{1-t}\right) -h(b)\right]
d\left( f\left( b^{t}G^{1-t}\right) \right)
\end{equation*}%
\begin{equation*}
=\left. \left[ 2h\left( b^{t}G^{1-t}\right) -h(b)\right] f\left(
b^{t}G^{1-t}\right) \right\vert _{0}^{1}
\end{equation*}%
\begin{equation*}
\ \ \ \ \ \ \ \ \ \ \ \ \ -2\ln \left( \frac{b}{G}\right)
\int\limits_{0}^{1}f\left( a^{t}G^{1-t}\right) h^{\prime }\left(
b^{t}G^{1-t}\right) b^{t}G^{1-t}dt
\end{equation*}

Therefore

\begin{equation}
\frac{I_{1}+I_{2}}{2}=\left[ h(b)-2h(a)\right] \frac{f(a)}{2}+h(b)\frac{f(b)%
}{2}-\frac{lnb-lna}{2}\left\{ \int\limits_{0}^{1}f\left( a^{t}G^{1-t}\right)
h^{\prime }\left( a^{t}G^{1-t}\right) a^{t}G^{1-t}dt\right.
\end{equation}%
\begin{equation*}
\ \ \ \ \ \ \ \ \ \ \ \ \ \ \ \ \ \ \ \ \ \ \ \ \ \ \ \ \ \ \ \ \ \ \ \ \ \
\ \ \ \ \ \ \ \ \ \ \ \ \ \ \ \ \ \ \ \ \ \ \ \ \ \ \ \ \ \ \ \ \ \ \ \ \ \
\ \ \ \ \ \ \ \ \ \ \ \ \ \ \ \ \ \ \ \ \ \ \ \ \ \left.
+\int\limits_{0}^{1}f\left( a^{t}G^{1-t}\right) h^{\prime }\left(
a^{t}G^{1-t}\right) a^{t}G^{1-t}dt\right\}
\end{equation*}

This complete the proof
\end{proof}

\begin{lemma}
For $a, G, b>0$, we have

\begin{equation}
\zeta _{1}\left( a,b\right) =\int\limits_{0}^{1} ta^{t}G^{1-t}\left\vert
2h\left( a^{t}G^{1-t}\right) -h(b)\right\vert dt
\end{equation}%
\begin{equation}
\zeta _{2}\left( a,b\right) =\int\limits_{0}^{1} (1-t)a^{t}G^{1-t}\left\vert
2h\left( a^{t}G^{1-t}\right) -h(b)\right\vert
dt+\int\limits_{0}^{1}(1-t)b^{t}G^{1-t}\left\vert 2h\left(
b^{t}G^{1-t}\right) -h(b)\right\vert dt
\end{equation}%
\begin{equation}
\zeta _{3}\left( a,b\right) =\int\limits_{0}^{1} tb^{t}G^{1-t}\left\vert
2h\left( b^{t}G^{1-t}\right) -h(b)\right\vert dt
\end{equation}
\end{lemma}

\begin{theorem}
Let $f:I\subseteq 
\mathbb{R}
=\left( 0,\infty \right) \longrightarrow 
\mathbb{R}
$ be differentiable mapping $I^{o}$, where $a,b\in I^{o}$ with $a<b$. If the
mapping $\left\vert f^{\prime }\right\vert $ is GA-convex on $\left[ a,b%
\right] $, then the following inequality holds:

\begin{equation}  \label{2.6}
\left\vert \left[ h(b)-2h(a)\right] \frac{f(a)}{2}+h(b)\frac{f(b)}{2}%
-\int\limits_{a}^{b}f(x)h^{\prime }(x)dx\right\vert
\end{equation}%
\begin{equation*}
\leq \frac{\ln b-\ln a}{4}\left[ \zeta _{1}(a,b)\left\vert f^{\prime
}(a)\right\vert +\zeta _{2}(a,b)\left\vert f^{\prime }(G)\right\vert +\zeta
_{3}(a,b)\left\vert f^{\prime }(b)\right\vert \right]
\end{equation*}
where $\zeta _{1}\left( a,b\right), \zeta _{2}\left( a,b\right), \zeta
_{3}\left( a,b\right)$ are defined in Lemma 2.
\end{theorem}

\begin{proof}
Continuing inequality (\ref{2.1}) in Lemma 1 
\begin{equation}  \label{2.7}
\left\vert \left[ h(b)-2h(a)\right] \frac{f(a)}{2}+h(b)\frac{f(b)}{2}%
-\int\limits_{a}^{b}f(x)h^{\prime }(x)dx\ \right\vert
\end{equation}%
\begin{equation*}
\leq \frac{lnb-lna}{4}\left\{ \int\limits_{0}^{1}\left\vert 2h\left(
a^{t}G^{1-t}\right) -h(b)\right\vert \left\vert f^{\prime }\left(
a^{t}G^{1-t}\right) a^{t}G^{1-t}\right\vert dt\ \right.
\end{equation*}%
\begin{equation*}
\ \ \ \ \ \ \ \ \ \ \ \ \ \ \ \ \ \ \left. +\int\limits_{0}^{1}\left\vert
2h\left( b^{t}G^{1-t}\right) -h(b)\right\vert \left\vert f^{\prime }\left(
b^{t}G^{1-t}\right) b^{t}G^{1-t}\right\vert dt\right\}
\end{equation*}

Using $\left\vert f^{\prime} \right\vert$ is GA-convex in (\ref{2.7})

\begin{equation}
\left\vert \left[ h(b)-2h(a)\right] \frac{f(a)}{2}+h(b)\frac{f(b)}{2}%
-\int\limits_{a}^{b}f(x)h^{\prime }(x)dx\ \right\vert  \label{2.8}
\end{equation}%
\begin{equation*}
\leq \frac{lnb-lna}{4}\left\{ \int\limits_{0}^{1}\left\vert 2h\left(
a^{t}G^{1-t}\right) -h(b)\right\vert \left[ t\left\vert f^{\prime
}(a)\right\vert +(1-t)\left\vert f^{\prime }(G)\right\vert \right]
a^{t}G^{1-t}dt\ \right.
\end{equation*}%
\begin{equation*}
\ \ \ \ \ \ \ \ \ \ \ \ \ \ \ \ \ \ \left. +\int\limits_{0}^{1}\left\vert
2h\left( b^{t}G^{1-t}\right) -h(b)\right\vert \left[ t\left\vert f^{\prime
}(b)\right\vert +(1-t)\left\vert f^{\prime }(G)\right\vert \right]
dt\right\} ,
\end{equation*}

by (\ref{2.8}) and Lemma 2, this proof is complete.
\end{proof}

\begin{corollary}
$g:\left[a,b\right]\longrightarrow \left[0,\infty\right)$ be continuous
positive mapping and symmetric to $\sqrt{ab}$. Let $h(x)=\int\limits_{a}^{x}%
\left[ \left( ln\frac{b}{t}\right) ^{\alpha -1}+\left( ln\frac{t}{a}\right)
^{\alpha -1}\right] \frac{g(t)}{t}dt$ for all $t\in \lbrack a,b]$ and $%
\alpha >0$ in Teorem 5, we obtain:

\begin{equation}
\left\vert \left( \frac{f(a)+f(b)}{2}\right) \left[ J_{a^{+}}^{\alpha
}g(b)+J_{b^{-}}^{\alpha }g(a)\right] -\left[ J_{a^{+}}^{\alpha }\left(
fg\right) \left( b\right) +J_{b^{-}}^{\alpha }\left( fg\right) \left(
a\right) \right] \right\vert  \label{2.9}
\end{equation}%
\begin{equation*}
\leq \frac{\left( \ln b-\ln a\right) ^{\alpha +1}}{2^{\alpha +1}\Gamma
\left( \alpha +1\right) }\left\Vert g\right\Vert _{\infty }\left[
C_{1}\left( \alpha \right) \left\vert f^{\prime }(a)\right\vert +C_{2}\left(
\alpha \right) \left\vert f^{\prime }(G)\right\vert +C_{3}\left( \alpha
\right) \left\vert f^{\prime }(b)\right\vert \right]
\end{equation*}

where 
\begin{equation*}
C_{1}\left( \alpha \right) =\int\limits_{0}^{1}\left[ (1+t)^{\alpha
}-(1-t)^{\alpha }\right] ta^{t}G^{1-t}dt
\end{equation*}%
\begin{equation*}
C_{2}\left( \alpha \right) =\int\limits_{0}^{1}\left( 1-t\right) \left[
(1+t)^{\alpha }-(1-t)^{\alpha }\right] \left[ a^{t}G^{1-t}+b^{t}G^{1-t}%
\right] dt
\end{equation*}%
\begin{equation*}
C_{3}\left( \alpha \right) =\int\limits_{0}^{1}\left[ (1+t)^{\alpha
}-(1-t)^{\alpha }\right] tb^{t}G^{1-t}dt
\end{equation*}

Specially in (\ref{2.9}) and using Lemma 2, for $0< \alpha \leq 1 $ we have:

\begin{equation}  \label{2.10}
\left\vert \left( \frac{f(a)+f(b)}{2}\right) \left[ J_{a^{+}}^{\alpha
}g(b)+J_{b^{-}}^{\alpha }g(a)\right] -\left[ J_{a^{+}}^{\alpha }\left(
fg\right) \left( b\right) +J_{b^{-}}^{\alpha }\left( fg\right) \left(
a\right) \right] \right\vert
\end{equation}%
\begin{equation*}
\leq \frac{\left( \ln b-\ln a\right) ^{\alpha +1}}{2 \Gamma \left( \alpha +1
\right) }\left\Vert g\right\Vert _{\infty }\left[ C_{1}\left( \alpha \right)
\left\vert f^{\prime }(a)\right\vert +C_{2}\left( \alpha \right) \left\vert
f^{\prime }(G)\right\vert +C_{3}\left( \alpha \right) \left\vert f^{\prime
}(b)\right\vert \right]
\end{equation*}

where 
\begin{equation*}
C_{1}\left( \alpha \right) =\int\limits_{0}^{1} t^{\alpha +1}a^{t}G^{1-t}
\end{equation*}%
\begin{equation*}
C_{2}\left( \alpha \right) =\int\limits_{0}^{1}\left[ \left( 1-t\right)
t^{\alpha }a^{t}G^{1-t}+\left( 1-t\right) t^{\alpha }b^{t}G^{1-t} \right]dt
\end{equation*}%
\begin{equation*}
C_{3}\left( \alpha \right) =\int\limits_{0}^{1} t^{\alpha +1}b^{t}G^{1-t}
\end{equation*}
\end{corollary}

\begin{proof}
By left side of inequality (\ref{2.8}) in Teorem 5, when we write $%
h(x)=\int\limits_{a}^{x}\left[ \left( ln\frac{b}{t}\right) ^{\alpha
-1}+\left( ln\frac{t}{a}\right) ^{\alpha -1}\right] \frac{g(t)}{t}dt$ for
all $t\in \lbrack a,b]$, we have

\begin{equation*}
\left\vert \Gamma \left( \alpha \right) \left( \frac{f(a)+f(b)}{2}\right) %
\left[ J_{a^{+}}^{\alpha }g(b)+J_{b^{-}}^{\alpha }g(a)\right] -\Gamma \left(
\alpha \right)\left[ J_{a^{+}}^{\alpha }\left( fg\right) \left( b\right)
+J_{b^{-}}^{\alpha }\left( fg\right) \left( a\right) \right] \right\vert
\end{equation*}%
.

On the other hand, right side of inequality (\ref{2.8})

\begin{equation}
\leq \frac{lnb-lna}{4}\left\{ \int\limits_{0}^{1}\left\vert 
\begin{array}{c}
2\int\limits_{a}^{a^{t}G^{1-t}}\left[ \left( \ln \frac{b}{x}\right) ^{\alpha
-1}+\left( \ln \frac{x}{a}\right) ^{\alpha -1}\right] \frac{g(x)}{x}dx \\ 
-\int\limits_{a}^{b}\left[ \left( \ln \frac{b}{x}\right) ^{\alpha -1}+\left(
\ln \frac{x}{a}\right) ^{\alpha -1}\right] \frac{g(x)}{x}dx%
\end{array}%
\right\vert \left[ t\left\vert f^{\prime }(a)\right\vert +(1-t)\left\vert
f^{\prime }(G)\right\vert \right] a^{t}G^{1-t}dt\ \right.  \label{2.11}
\end{equation}%
\begin{equation*}
+\left. \int\limits_{0}^{1}\left\vert 
\begin{array}{c}
2\int\limits_{a}^{b^{t}G^{1-t}}\left[ \left( \ln \frac{b}{x}\right) ^{\alpha
-1}+\left( \ln \frac{x}{a}\right) ^{\alpha -1}\right] \frac{g(x)}{x}dx \\ 
-\int\limits_{a}^{b}\left[ \left( \ln \frac{b}{x}\right) ^{\alpha -1}+\left(
\ln \frac{x}{a}\right) ^{\alpha -1}\right] \frac{g(x)}{x}dx%
\end{array}%
\right\vert \left[ t\left\vert f^{\prime }(b)\right\vert +(1-t)\left\vert
f^{\prime }(G)\right\vert \right] b^{t}G^{1-t}dt\right\}
\end{equation*}%
.

Since $g(x)$ is symmetric to $x=\sqrt{ab} $, we have 
\begin{equation}  \label{2.12}
\left\vert 2\int\limits_{a}^{a^{t}G^{1-t}}\left[ \left( \ln \frac{b}{x}%
\right) ^{\alpha -1}+\left( \ln \frac{x}{a}\right) ^{\alpha -1}\right] \frac{%
g(x)}{x}dx-\int\limits_{a}^{b}\left[ \left( \ln \frac{b}{x}\right) ^{\alpha
-1}+\left( \ln \frac{x}{a}\right) ^{\alpha -1}\right] \frac{g(x)}{x}%
dx\right\vert
\end{equation}%
\begin{equation*}
\ \ \ \ \ \ \ \ \ \ \ \ \ \ \ \ \ \ \ \ \ \ \ \ \ \ \ \ \ \ \ \ \ \ \ \ \ \
\ \ \ \ \ \ \ \ \ \ \ \ \ \ \ \ \ \ \ \ \ \ \ \ \ \ \ \ \ \ \ \ \ \ \ \ \ \
\ \ \ \ \ \ \ \ \ \ \ \ \ \ \ =\left\vert
\int\limits_{a^{t}G^{1-t}}^{b^{t}G^{1-t}}\left[ \left( \ln \frac{b}{x}%
\right) ^{\alpha -1}+\left( \ln \frac{x}{a}\right) ^{\alpha -1}\right] \frac{%
g(x)}{x}dx\right\vert
\end{equation*}
and

\begin{equation}  \label{2.13}
\left\vert 2\int\limits_{a}^{b^{t}G^{1-t}}\left[ \left( \ln \frac{b}{x}%
\right) ^{\alpha -1}+\left( \ln \frac{x}{a}\right) ^{\alpha -1}\right] \frac{%
g(x)}{x}dx-\int\limits_{a}^{b}\left[ \left( \ln \frac{b}{x}\right) ^{\alpha
-1}+\left( \ln \frac{x}{a}\right) ^{\alpha -1}\right] \frac{g(x)}{x}%
dx\right\vert
\end{equation}%
\begin{equation*}
\ \ \ \ \ \ \ \ \ \ \ \ \ \ \ \ \ \ \ \ \ \ \ \ \ \ \ \ \ \ \ \ \ \ \ \ \ \
\ \ \ \ \ \ \ \ \ \ \ \ \ \ \ \ \ \ \ \ \ \ \ \ \ \ \ \ \ \ \ \ \ \ \ \ \ \
\ \ \ \ \ \ \ \ \ \ \ \ \ \ \ =\left\vert
\int\limits_{a^{t}G^{1-t}}^{b^{t}G^{1-t}}\left[ \left( \ln \frac{b}{x}%
\right) ^{\alpha -1}+\left( \ln \frac{x}{a}\right) ^{\alpha -1}\right] \frac{%
g(x)}{x}dx\right\vert
\end{equation*}

for all $t\in [0,1] $.

By (\ref{2.10}), (\ref{2.11}) and (\ref{2.12}), we have

\begin{equation}
\left\vert \left( \frac{f(a)+f(b)}{2}\right) \left[ J_{a^{+}}^{\alpha
}g(b)+J_{b^{-}}^{\alpha }g(a)\right] -\left[ J_{a^{+}}^{\alpha }\left(
fg\right) \left( b\right) +J_{b^{-}}^{\alpha }\left( fg\right) \left(
a\right) \right] \right\vert  \label{2.14}
\end{equation}%
\begin{equation*}
\leq \frac{lnb-lna}{4\Gamma \left( \alpha \right) }\left\{
\int\limits_{0}^{1}\left\vert \int\limits_{a^{t}G^{1-t}}^{b^{t}G^{1-t}}\left[
\left( \ln \frac{b}{x}\right) ^{\alpha -1}+\left( \ln \frac{x}{a}\right)
^{\alpha -1}\right] \frac{g(x)}{x}dx\right\vert \left[ t\left\vert f^{\prime
}(a)\right\vert +(1-t)\left\vert f^{\prime }(G)\right\vert \right]
a^{t}G^{1-t}dt\ \right.
\end{equation*}%
\begin{equation*}
\ \ \ \ \ \ \ \ \ \ \ \ \ \ \ \ \ \ \ \ \ \ \ \ \ +\left.
\int\limits_{0}^{1}\left\vert \int\limits_{a^{t}G^{1-t}}^{b^{t}G^{1-t}}\left[
\left( \ln \frac{b}{x}\right) ^{\alpha -1}+\left( \ln \frac{x}{a}\right)
^{\alpha -1}\right] \frac{g(x)}{x}dx\right\vert \left[ t\left\vert f^{\prime
}(b)\right\vert +(1-t)\left\vert f^{\prime }(G)\right\vert \right]
b^{t}G^{1-t}dt\right\}
\end{equation*}%
\begin{equation*}
\leq \frac{lnb-lna}{4\Gamma \left( \alpha \right) }\left\Vert g\right\Vert
_{\infty }\left\{ \int\limits_{0}^{1}\left[ \int%
\limits_{a^{t}G^{1-t}}^{b^{t}G^{1-t}}\left[ \left( \ln \frac{b}{x}\right)
^{\alpha -1}+\left( \ln \frac{x}{a}\right) ^{\alpha -1}\right] \frac{1}{x}dx%
\right] \left[ t\left\vert f^{\prime }(a)\right\vert +(1-t)\left\vert
f^{\prime }(G)\right\vert \right] a^{t}G^{1-t}dt\ \right.
\end{equation*}%
\begin{equation*}
\ \ \ \ \ \ \ \ \ \ \ \ \ \ \ \ \ \ \ \ \ \ \ \ \ +\left.
\int\limits_{0}^{1} \left[ \int\limits_{a^{t}G^{1-t}}^{b^{t}G^{1-t}}\left[
\left( \ln \frac{b}{x}\right) ^{\alpha -1}+\left( \ln \frac{x}{a}\right)
^{\alpha -1}\right] \frac{1}{x}dx\right] \left[ t\left\vert f^{\prime
}(b)\right\vert +(1-t)\left\vert f^{\prime }(G)\right\vert \right]
b^{t}G^{1-t}dt\right\} .
\end{equation*}%
In the last inequality,

\begin{equation}
\int\limits_{a^{t}G^{1-t}}^{b^{t}G^{1-t}}\left[ \left( \ln \frac{b}{x}%
\right) ^{\alpha -1}+\left( \ln \frac{x}{a}\right) ^{\alpha -1}\right] \frac{%
1}{x}dx=\int\limits_{a^{t}G^{1-t}}^{b^{t}G^{1-t}}\left( \ln \frac{b}{x}%
\right) ^{\alpha -1}\ \frac{1}{x}dx+\int%
\limits_{a^{t}G^{1-t}}^{b^{t}G^{1-t}}\left( \ln \frac{x}{a}\right) ^{\alpha
-1}\frac{1}{x}dx  \label{2.15}
\end{equation}%
\begin{equation*}
\ \ \ \ \ \ \ \ \ \ \ \ \ \ \ \ \ \ \ \ \ \ \ \ \ \ \ \ \ \ \ \ \ \ \ \ \ \
\ \ \ \ \ \ \ \ \ \ \ \ \ \ \ \ \ =\frac{2.\left( \ln b-\ln a\right)
^{\alpha }}{2^{\alpha }.\alpha }\left[ \left( 1+t\right) ^{\alpha }-\left(
1-t\right) ^{\alpha }\right] .
\end{equation*}

By Lemma 1, we have 
\begin{equation*}
\int\limits_{a^{t}G^{1-t}}^{b^{t}G^{1-t}}\left[ \left( \ln \frac{b}{x}%
\right) ^{\alpha -1}+\left( \ln \frac{x}{a}\right) ^{\alpha -1}\right] \frac{%
1}{x}dx=\int\limits_{a^{t}G^{1-t}}^{b^{t}G^{1-t}}\left( \ln \frac{b}{x}%
\right) ^{\alpha -1}\ \frac{1}{x}dx+\int%
\limits_{a^{t}G^{1-t}}^{b^{t}G^{1-t}}\left( \ln \frac{x}{a}\right) ^{\alpha
-1}\frac{1}{x}dx
\end{equation*}%
\begin{equation*}
\ \ \ \ \ \ \ \ \ \ \ \ \ \ \ \ \ \ \ \ \ \ \ \ \ \ \ \ \ \ \ \ \ \ \ \ \ \
\ \ \ \ \ \ \ \ \ \ \ \ \ \ \ \ \ \leq \frac{2.\left( \ln b-\ln a\right)
^{\alpha }}{\alpha }t^{\alpha }
\end{equation*}

A combination of (\ref{2.13}) and (\ref{2.14}), we have (\ref{2.9}). This
complete is proof.
\end{proof}

\begin{corollary}
In Corollary 1,

(1)If $\alpha =1$ is in corollary, we obtain following
Hermite-Hadamard-Fejer Type inequality for GA-convex function which is
related in (\ref{2.10}):

\begin{equation}
\left\vert \left[ \frac{f(a)+f(b)}{2}\right] \underset{a}{\overset{b}{\int }}%
\frac{g(x)}{x}dx-\underset{a}{\overset{b}{\int }}f(x)\frac{g(x)}{x}%
dx\right\vert \leq  \label{2.16}
\end{equation}%
\begin{equation*}
\frac{\left( \ln b-\ln a\right) ^{2}}{4}\left\Vert g\right\Vert _{\infty }%
\left[ C_{1}(1)\left\vert f^{\prime }\left( a\right) \right\vert
+C_{2}(1)\left\vert f^{\prime }\left( G\right) \right\vert
+C_{3}(1)\left\vert f^{\prime }\left( b\right) \right\vert \right]
\end{equation*}%
where for $a,b,G>0$, we have 
\begin{equation*}
C_{1}(1)=\overset{1}{\underset{0}{\int }}t^{2}a^{t}G^{1-t}dt=\frac{2}{\ln
b-\ln a}\left\{ -a-\frac{4a}{\ln b-\ln a}-\frac{8a-8G}{\left( \ln b-\ln
a\right) ^{2}}\right\} ,
\end{equation*}%
\begin{equation*}
C_{2}\left( 1\right) =\overset{1}{\underset{0}{\int }}t\left( 1-t\right)
a^{t}G^{1-t}dt+\overset{1}{\underset{0}{\int }}t\left( 1-t\right)
b^{t}G^{1-t}dt=\frac{2}{\ln b-\ln a}\left\{ \frac{2\left( a+b+2G\right) }{%
\ln b-\ln a}+\frac{8\left( a-b\right) }{\left( \ln b-\ln a\right) ^{2}}%
\right\} ,
\end{equation*}%
\qquad \qquad and%
\begin{equation*}
C_{3}(1)=\overset{1}{\underset{0}{\int }}t^{2}a^{t}G^{1-t}dt=\frac{2}{\ln
b-\ln a}\left\{ b-\frac{4b}{\ln b-\ln a}+\frac{8b-8G}{\left( \ln b-\ln
a\right) ^{2}}\right\} .
\end{equation*}

(2)If $g(x)=1$ is in corollary, we obtain following Hermite-Hadamard-Fejer
Type inequality for GA-convex function which is related in (\ref{2.9}):

\begin{equation}
\left\vert \left( \frac{f(a)+f(b)}{2}\right) -\frac{\Gamma (\alpha +1)}{%
2(\ln b-\ln a)^{\alpha }}\left[ J_{a^{+}}^{\alpha }f\left( b\right)
+J_{b^{-}}^{\alpha }f\left( a\right) \right] \right\vert  \label{2.17}
\end{equation}%
\begin{equation*}
\leq \frac{\left( \ln b-\ln a\right) }{2^{\alpha +2}}\left[ C_{1}\left(
\alpha \right) \left\vert f^{\prime }(a)\right\vert +C_{2}\left( \alpha
\right) \left\vert f^{\prime }(G)\right\vert +C_{3}\left( \alpha \right)
\left\vert f^{\prime }(b)\right\vert \right] .
\end{equation*}

(3)If $g(x)=1$ and $\alpha =1$ is in corollary, we obtain following
Hermite-Hadamard-Fejer Type inequality for GA-convex function which is
related in (\ref{2.10}):

\begin{equation}
\left\vert \left( \frac{f(a)+f(b)}{2}\right) -\frac{1}{(\ln b-\ln a)}%
\int\limits_{a}^{b}\frac{f(x)}{x}dx\right\vert  \label{2.18}
\end{equation}%
\begin{equation*}
\leq \frac{\left( \ln b-\ln a\right) }{4}\left[ C_{1}\left( 1\right)
\left\vert f^{\prime }(a)\right\vert +C_{2}\left( 1\right) \left\vert
f^{\prime }(G)\right\vert +C_{3}\left( 1\right) \left\vert f^{\prime
}(b)\right\vert \right] .
\end{equation*}
\end{corollary}

\begin{theorem}
Let $f:I\subseteq 
\mathbb{R}
^{+}=\left( 0,\infty \right) \longrightarrow 
\mathbb{R}
$ be differentiable mapping $I^{o}$, where $a,b\in I$ with $a<b$, and $g:%
\left[ a,b\right] \longrightarrow \left[ 0,\infty \right) $ be continuous
positive mapping and symmetric to $\sqrt{ab}$ and $\frac{1}{q}+\frac{1}{p}=1$%
. If the mapping $\left\vert f^{\prime }\right\vert ^{q}$ is GA-convex on $%
\left[ a,b\right] $, then the following inequality holds:

\begin{equation}  \label{2.19}
\left\vert \left[ h(b)-2h(a)\right] \frac{f(a)}{2}+h(b)\frac{f(b)}{2}%
-\int\limits_{a}^{b}f(x)h^{\prime }(x)dx\right\vert
\end{equation}%
\begin{equation*}
\leq \frac{\ln b-\ln a}{4}\left\{ 
\begin{array}{c}
\left( \underset{0}{\int\limits^{1}}\left\vert
2h(a^{t}G^{1-t})-h(b)\right\vert dt\right) ^{1-\frac{1}{q}}\times \\ 
\left( \underset{0}{\overset{1}{\int }}%
\begin{array}{c}
\left( \left\vert 2h(a^{t}G^{1-t})-h(b)\right\vert dt\right) \\ 
\times \left( ta^{qt}G^{q(1-t)}\left\vert f^{\prime }(a)\right\vert
^{q}+(1-t)a^{qt}G^{q(1-t)}\left\vert f^{\prime }(G)\right\vert ^{q}\right)%
\end{array}%
\right) ^{\frac{1}{q}}%
\end{array}%
\right.
\end{equation*}%
\begin{equation*}
\ \ \ \ \ \ \ \ \ \ \ \ \ \ \ \ \ \ \ \ \ \ \ \ \ \ \ \ \ \ \ \ \ \ \ \ \ \
\left. 
\begin{array}{c}
+\ \left( \underset{0}{\int\limits^{1}}\left\vert
2h(b^{t}G^{1-t})-h(b)\right\vert dt\right) ^{1-\frac{1}{q}}\times \\ 
\left( \underset{0}{\overset{1}{\int }}%
\begin{array}{c}
\left( \left\vert 2h(b^{t}G^{1-t})-h(b)\right\vert dt\right) \\ 
\times \left( tb^{qt}G^{q(1-t)}\left\vert f^{\prime }(b)\right\vert
^{q}+(1-t)b^{qt}G^{q(1-t)}\left\vert f^{\prime }(G)\right\vert ^{q}\right)%
\end{array}%
\right) ^{\frac{1}{q}}%
\end{array}%
\right\}
\end{equation*}
\end{theorem}

\begin{proof}
Continuing from (\ref{2.7}) in Theorem 5, we use Holder Inequality and we
use that $\left\vert f^{\prime}\right\vert ^{q} $ is GA-convex. Thus this
proof is complete.
\end{proof}

\begin{corollary}
Let $h(x)=\int\limits_{a}^{x}\left[ \left( ln\frac{b}{t}\right) ^{\alpha
-1}+\left( ln\frac{t}{a}\right) ^{\alpha -1}\right] \frac{g(t)}{t}dt$ for
all $t\in \lbrack a,b]$ in Teorem 6, we obtain:

\begin{equation}  \label{2.20}
\left\vert \left( \frac{f(a)+f(b)}{2}\right) \left[ J_{a^{+}}^{\alpha
}g(b)+J_{b^{-}}^{\alpha }g(a)\right] -\left[ J_{a^{+}}^{\alpha }\left(
fg\right) \left( b\right) +J_{b^{-}}^{\alpha }\left( fg\right) \left(
a\right) \right] \right\vert
\end{equation}

\begin{equation*}
\leq \frac{\left( \ln b-\ln a\right) ^{\alpha +1}\left\Vert g\right\Vert
_{\infty }}{2^{\alpha +1}\Gamma \left( \alpha +1\right) }(\frac{2^{\alpha
+2}-2^{2}}{\alpha +1})^{1-\frac{1}{q}}\left[ C_{1}\left( \alpha ,q\right)
\left\vert f^{\prime }(a)\right\vert ^{q}+C_{2}\left( \alpha ,q\right)
\left\vert f^{\prime }(G)\right\vert ^{q}+C_{3}\left( \alpha ,q\right)
\left\vert f^{\prime }(b)\right\vert ^{q}\right] ^{\frac{1}{q}}
\end{equation*}%
where for $q>1$

\begin{equation*}
C_{1}\left( \alpha ,q\right) =\overset{1}{\underset{0}{\int }}\left[
(1+t)^{\alpha }-(1-t)^{\alpha }\right] ta^{qt}G^{q(1-t)}dt
\end{equation*}%
\begin{equation*}
C_{2}\left( \alpha ,q\right) =\overset{1}{\underset{0}{\int }}\left[
(1+t)^{\alpha }-(1-t)^{\alpha }\right] (1-t)\left(
a^{qt}G^{q(1-t)}+b^{qt}G^{q(1-t)}\right) dt
\end{equation*}%
\begin{equation*}
C_{3}\left( \alpha ,q\right) =\overset{1}{\underset{0}{\int }}\left[
(1+t)^{\alpha }-(1-t)^{\alpha }\right] tb^{qt}G^{q(1-t)}dt.
\end{equation*}
\end{corollary}

\begin{proof}
Continuing from (\ref{2.15}) of Corollary 1 and (\ref{2.19}) in Theorem 6,

\begin{equation}  \label{2.21}
\left\vert \left( \frac{f(a)+f(b)}{2}\right) \left[ J_{a^{+}}^{\alpha
}g(b)+J_{b^{-}}^{\alpha }g(a)\right] -\left[ J_{a^{+}}^{\alpha }\left(
fg\right) \left( b\right) +J_{b^{-}}^{\alpha }\left( fg\right) \left(
a\right) \right] \right\vert
\end{equation}

\begin{equation*}
\leq \frac{\left( \ln b-\ln a\right) ^{\alpha +1}\left\Vert g\right\Vert
_{\infty }}{2^{\alpha +1}\Gamma \left( \alpha +1\right) }\left\{ 
\begin{array}{c}
\left( \underset{0}{\overset{1}{\int }}\left[ (1+t)^{\alpha }-(1-t)^{\alpha }%
\right] dt\right) ^{1-\frac{1}{q}}\times \\ 
\left( \underset{0}{\overset{1}{\int }}\left[ (1+t)^{\alpha }-(1-t)^{\alpha }%
\right] \left( ta^{qt}G^{q(1-t)}\left\vert f^{\prime }(a)\right\vert
^{q}+(1-t)a^{qt}G^{q(1-t)}\left\vert f^{\prime }(G)\right\vert ^{q}\right)
dt\right) ^{\frac{1}{q}}%
\end{array}%
\right.
\end{equation*}%
\begin{equation*}
\ \ \ \ \ \ \ \ \ \ \ \ \ \ \ \ \ \ \ \ \ \ \ \ \ \ \ \ \ \ \ \ +\left. 
\begin{array}{c}
\left( \underset{0}{\overset{1}{\int }}\left[ (1+t)^{\alpha }-(1-t)^{\alpha }%
\right] dt\right) ^{1-\frac{1}{q}}\times \\ 
\left( \underset{0}{\overset{1}{\int }}\left[ (1+t)^{\alpha }-(1-t)^{\alpha }%
\right] \left( tb^{qt}G^{q(1-t)}\left\vert f^{\prime }(b)\right\vert
^{q}+(1-t)b^{qt}G^{q(1-t)}\left\vert f^{\prime }(G)\right\vert ^{q}\right)
dt\right) ^{\frac{1}{q}}%
\end{array}%
\right\}
\end{equation*}%
\begin{equation*}
\leq \frac{\left( \ln b-\ln a\right) ^{\alpha +1}\left\Vert g\right\Vert
_{\infty }}{2^{\alpha +1}\Gamma \left( \alpha +1\right) }(\frac{2^{\alpha
+1}-2}{\alpha +1})^{1-\frac{1}{q}}\left[ 
\begin{array}{c}
\left( \underset{0}{\overset{1}{\int }}%
\begin{array}{c}
\left[ (1+t)^{\alpha }-(1-t)^{\alpha }\right] \times \\ 
\left[ ta^{qt}G^{q(1-t)}\left\vert f^{\prime }(a)\right\vert
^{q}+(1-t)a^{qt}G^{q(1-t)}\left\vert f^{\prime }(G)\right\vert ^{q}\right]%
\end{array}%
dt\right) ^{\frac{1}{q}} \\ 
+\left( \underset{0}{\overset{1}{\int }}%
\begin{array}{c}
\left[ (1+t)^{\alpha }-(1-t)^{\alpha }\right] \times \\ 
\left[ tb^{qt}G^{q(1-t)}\left\vert f^{\prime }(b)\right\vert
^{q}+(1-t)b^{qt}G^{q(1-t)}\left\vert f^{\prime }(G)\right\vert ^{q}\right]%
\end{array}%
dt\right) ^{\frac{1}{q}}%
\end{array}%
\right]
\end{equation*}

By the power-mean inequality $\left( a^{r}+b^{r}<2^{1-r}\left( a+b\right)
^{r}for\quad a>0,b>0,\quad r<1\right) $ and $\frac{1}{q}+\frac{1}{q}=1$ we
have

\begin{equation}
\leq \frac{\left( \ln b-\ln a\right) ^{\alpha +1}\left\Vert g\right\Vert
_{\infty }}{2^{\alpha +1}\Gamma \left( \alpha +1\right) }(\frac{2^{\alpha
+2}-2^{2}}{\alpha +1})^{1-\frac{1}{q}}
\end{equation}%
\begin{equation*}
\ \ \ \ \ \ \ \ \ \ \ \ \ \ \ \ \ \ \ \ \ \ \ \ \ \ \ \ \ \ \ \ \ \ \ \ \ \
\ \ \ \ \ \ \ \ \ \ \times \ \left[ \overset{1}{\underset{0}{\int }}\left( 
\begin{array}{c}
\left[ (1+t)^{\alpha }-(1-t)^{\alpha }\right] ta^{qt}G^{q(1-t)}\left\vert
f^{\prime }(a)\right\vert ^{q}+ \\ 
\left[ (1+t)^{\alpha }-(1-t)^{\alpha }\right] (1-t)\left( 
\begin{array}{c}
a^{qt}G^{q(1-t)} \\ 
+b^{qt}G^{q(1-t)}%
\end{array}%
\right) \left\vert f^{\prime }(G)\right\vert ^{q} \\ 
+\left[ (1+t)^{\alpha }-(1-t)^{\alpha }\right] tb^{qt}G^{q(1-t)}\left\vert
f^{\prime }(b)\right\vert ^{q}%
\end{array}%
\right) dt\right] ^{\frac{1}{q}}
\end{equation*}
\end{proof}

\begin{corollary}
When $\alpha =1$ and $g(x)=\frac{1}{\ln b-\ln a}$ is taken in Corollary 3,
we obtain:

\begin{equation}
\left\vert \left( \frac{f(a)+f(b)}{2}\right) -\frac{1}{\ln b-\ln a}\overset{b%
}{\underset{a}{\int }}\frac{f(x)}{x}dx\right\vert  \label{2.23}
\end{equation}

\begin{equation*}
\leq \frac{\left( \ln b-\ln a\right) }{2^{1+\frac{1}{q}}}\left[ C_{1}\left(
1,q\right) \left\vert f^{\prime }(a)\right\vert ^{q}+C_{2}\left( 1,q\right)
\left\vert f^{\prime }(G)\right\vert ^{q}+C_{3}\left( 1,q\right) \left\vert
f^{\prime }(b)\right\vert ^{q}\right] ^{\frac{1}{q}}.
\end{equation*}

This proof is complete.
\end{corollary}

\end{document}